\documentclass[a4paper,12pt,oneside]{amsart}       

\usepackage[british,english]{babel} 

\theoremstyle{cited}
\newtheorem{teor}{Theorem}[section]
\newtheorem{lem}[teor]{Lemma}
\newtheorem{cor}[teor]{Corollary}
\newtheorem{prop}[teor]{Proposition}
\newtheorem{conj}[teor]{Conjecture}

\theoremstyle{definition}
\newtheorem{deft}[teor]{Definition}

\theoremstyle{remark}
\newtheorem{oss}[teor]{Remark}

\newcommand{\C}{\mathbb{C}}
\newcommand{\R}{\mathbb{R}}

\newcommand{\N}{\mathbb{N}}
\newcommand{\DD}{\boldsymbol\Delta}
\newcommand{\vol}{\textup{Vol}}

\title[Volume rigidity at ideal points]{Volume rigidity at ideal points of the character variety of hyperbolic
  $3$-manifolds}

\author{Stefano Francaviglia and Alessio Savini}

\begin{document}

\maketitle

\begin{abstract}
Given the fundamental group $\Gamma$ of a finite-volume complete hyperbolic $3$-manifold $M$,
it is possible to associate to any representation $\rho:\Gamma \rightarrow
\textup{Isom}(\mathbb{H}^3)$ a numerical invariant called volume. This invariant is bounded by
the hyperbolic volume of $M$ and satisfies a
rigidity condition: if the volume of $\rho$ is maximal, then $\rho$  must
be conjugated to the holonomy of the hyperbolic structure of $M$. This paper generalizes this
rigidity result by showing that if a sequence of representations of $\Gamma$ into
$\textup{Isom}(\mathbb{H}^3)$ satisfies $\lim_{n \to \infty} \textup{Vol}(\rho_n) =
\textup{Vol}(M)$, then there must exist a sequence of elements $g_n \in
\textup{Isom}(\mathbb{H}^3)$ such that the representations $g_n \circ \rho_n \circ g_n^{-1}$
converge to the holonomy of $M$. In particular if the sequence $(\rho_n)_{n \in \mathbb{N}}$ converges to an ideal point of the
character variety, then the sequence of volumes must stay away from the maximum.
In this way we give an answer to ~\cite[Conjecture 1]{guilloux:articolo}. We conclude by
generalizing the result to the case of $k$-manifolds and representations in
$\textup{Isom}(\mathbb H^m)$, where $m\geq k \geq 3$.
\end{abstract}

\vspace{20pt}


\section{Introduction}

Let $\Gamma$ be the fundamental group of a (non-compact) complete hyperbolic $3$-manifold $M$ with finite
volume (hence with toric cusps). The volume of a representation $\rho: \Gamma \rightarrow
PSL(2,\C)=\textup{Isom}^+(\mathbb{H}^3)$ can be defined in several ways. For instance, it can
be thought of as the integral of the pullback of the volume form on $\mathbb{H}^3$ along any
pseudo-developing map $D: \tilde M \rightarrow \mathbb{H}^3$, as written both
in~\cite{dunfield:articolo} and in~\cite{franc04:articolo}. Since the volume is indipendent of
the choice of the pseudeveloping map $D$, when $D$ is a straight map this notion is a
generalization of the volume of a solution for the gluing equations associated to a
triangulation of $M$, given for instance in~\cite{neumann:articolo}. Another way to define the
volume of a representation $\rho$ is based on the properties of the bounded cohomology of the
group $\textup{Isom}^+(\mathbb{H}^3)$. In~\cite{bucher2:articolo} the authors prove that the
volume class $\omega_3$ is a generator for the cohomology group
$H^3_{cb}(\textup{Isom}^+(\mathbb{H}^3))$, hence, starting from it, we can construct a class in
$H^3_b(\Gamma)$ by pulling back $\omega_3$ along $\rho_b^*$ and then evaluate this class with a
relative fundamental class $[N,\partial N] \in H^3(N,\partial N)$ via the Kronecker
pairing. Here $N$ is any compact core of $M$. The equivalence between the two different
definitions it is shown for example in~\cite{kim:articolo}. To extend the notion of volume to
the more general case of representations into the whole group of the isometries
$\textup{Isom}(\mathbb{H}^3)$ the approach of~\cite{franc06:articolo} is to consider the
infimum all over the values $\textup{Vol}(D)$, where $D:\tilde M \rightarrow \mathbb{H}^3$ is a
properly ending smooth $\rho$-equivariant map (the existence of such maps is proved in~\cite{franc06:articolo} as well).

Since the volume is invariant under conjugation by an element of
$\textup{Isom}(\mathbb{H}^3)$, there exists a well-defined volume function on the character
variety $X(\Gamma,\textup{Isom}(\mathbb{H}^3))$ which is continuous with respect to the
topology of pointwise convergence. Moreover, this function satisfies a well-known rigidity
condition. As written in~\cite{franc04:articolo} (see~\cite{franc06:articolo} for higher
dimensional cases), for
any representation $\rho$ we have that $\vol(\rho)\leq \vol(M)$ and if equality holds
we must have $\rho=gig^{-1}$ where $i$ is the standard
lattice embedding $i:\Gamma \rightarrow \textup{Isom}(\mathbb{H}^3)$ and $g \in
\textup{Isom}(\mathbb{H}^3)$. Beyond its intrinsic interest, this result has important
consequences for example in the study of the AJ-conjecture for hyperbolic knot manifolds, as
written in~\cite{le:articolo}.

By generalizing both~\cite{culler:articolo} and~\cite{morgan:articolo},
in~\cite{morgan2:articolo} the author proposed a compactification of the variety
$X(\Gamma,\textup{Isom}(\mathbb{H}^3))$ whose ideal points can be interpreted as projective
lenght functions of isometric $\Gamma$-actions on real trees. It is natural to ask if there exists a way to extend the volume function to this compactification and which are the possible values attained at any ideal point. For instance, one could ask if it is possible to extend the ridigity of volume also at ideal points. A similar problem has already been conjectured in~\cite{guilloux:articolo} relatively to the rigidity of the Borel function with respect to the ideal points of the Morgan--Shalen compactification of the character variety $X(\Gamma,PSL(n,\C))$. More precisely, the statement is 

\begin{conj}[{\upshape ~\cite{guilloux:articolo}}]
Let $M$ be an orientable cusped hyperbolic $3$-manifold. Let $X_n$ be the geometric component of the $PSL(n,\C)$-character variety and let $\textup{hol}_{\textup{per}}$ be the peripheral holonomy map. Then, outside a neighborhood of the geometric representation $[\rho_{\textup{geom}}]$ the Borel function is bounded away from its maximum on $X_n$. 
\end{conj}

In this paper we are going to prove the conjecture for representations into $\textup{Isom}(\mathbb{H}^3)$, hence in the particular case of $n=2$. Indeed we will prove the following

\begin{teor}\label{convergence}
Let $\Gamma$ be the fundamental group of a non-compact complete hyperbolic $3$-manifold $M$ with finite volume.
Let $\rho_n:\Gamma \rightarrow \textup{Isom}(\mathbb{H}^3)$ be a sequence of representations such that $\lim_{n \to \infty} \textup{Vol}(\rho_n)=\textup{Vol}(M)$. Then there must exist a sequence of elements $g_n \in \textup{Isom}(\mathbb{H}^3)$ such that the sequence $g_n \circ \rho_n \circ g_n^{-1}$ converges to the standard lattice embedding $i:\Gamma \rightarrow \textup{Isom}(\mathbb{H}^3)$.
\end{teor}

Which implies 

\begin{cor}\label{rigidity}
Suppose $\rho_n:\Gamma  \rightarrow \textup{Isom}(\mathbb{H}^3)$ is a sequence of representations converging to any ideal point of the Morgan--Shalen compactification of $X(\Gamma,\textup{Isom}(\mathbb{H}^3))$. Then the sequence of volumes $\textup{Vol}(\rho_n)$ must be bounded from above by $\textup{Vol}(M)-\varepsilon$ for a suitable $\varepsilon>0$. 
\end{cor}

We also prove the generalization of Theorem~\ref{convergence} to the case of $k$-manifolds and
representations into $\textup{Isom}(\mathbb H^m)$ with $m \geq k \geq 3$ (Theorem~\ref{convk>3}).

The interest of Theorem~\ref{convergence} relies on the fact that $\Gamma$ admits non-trivial deformations
inside $PSL(2,\mathbb{C})$. Indeed by both~\cite{thurston:libro} and~\cite{neumann:articolo} the component 
of the character variety $X(\Gamma,PSL(2,\mathbb{C}))$ containing the class of the standard lattice embedding has complex 
dimension equal to the number of cusps of $M$.

Similarly, when $m>k\geq 3$ the space of representations is rich. For instance one can bend along
geodesic hypersurfaces (see~\cite{Ap}), but also purely parabolic deformations are possible (see~\cite{FP} for the study of deformations in $\mathbb{H}^4$ of complements of hyperbolic two-bridge knots).

When $k=m>3$ something different happens. Garland and Raghunathan
showed in~\cite{garland:articolo} that if $\Gamma$ is a non-uniform lattice of $\textup{Isom}(\mathbb{H}^k)$ without torsion 
then it holds $H^1(\Gamma,Ad \circ i)=0$, where $i$ denotes the standard lattice embedding of $\Gamma$. This phenomenon is 
called infinitesimal rigidity and it implies that the class $[i]$ is isolated in the character variety $X(\Gamma,\textup{Isom}(\mathbb{H}^k))$ and hence
$\Gamma$ is locally rigid. The result of Garland and Raghunathan extended to non-uniform lattices the property of local rigidity, already known
for uniform lattices by~\cite{selberg:articolo,calabi:articolo} and~\cite{weil:articolo}. It is worth noticing that the local rigidity of $\Gamma$ when $k \geq 4$ together with Theorem~\ref{convk=m} will imply that the sequence $(\rho_n)_{n \in \mathbb{N}}$ must be eventually constant in the character variety.

The proof of Theorem~\ref{convergence} will be based essentially on the so-called BCG--natural map
associated to a non-elementary representation $\rho:\Gamma \rightarrow
\textup{Isom}(\mathbb{H}^3)$, described in~\cite{besson95:articolo},~\cite{besson96:articolo}
and~\cite{besson99:articolo}. Given such a representation there exists a map $F:\mathbb{H}^3
\rightarrow \mathbb{H}^3$ which is equivariant with respect to $\rho$, smooth and satisfies
$|\det D_xF| \leq 1$ for every $x \in \mathbb{H}^3$. Moreover, the equality holds if and only if
$D_xF$ is an isometry, and we will exploit the fact that this claim can be made $\varepsilon$-accurate
if $|\det D_xF|>1-\varepsilon$. These properties make the natural map $F$ a powerful tool in the study
of volume rigidity (see~\cite{souto:articolo} for this kind of application).

The reader who is not an expert of BCG techniques may wonder why we need to assume $k\geq 3$. The crucial points are
the estimate $|\det D_xF|\leq 1$ and the study of the equality case. The latter boils down to the
study of the function $\psi(H):=\det(H)/\det(I-H)^2$ defined on the space of symmetric, positive definite,
$k \times k$ matrices with trace equal to $1$, and the reader can check that the case $k=2$ and $k\geq 3$
differ dramatically (see Remark~\ref{rem:BCGAppB}). 

\medskip

The paper is structured as follows. The first section is dedicated to preliminary definitions. We briefly recall the notion of
barycentre of a positive Borel measure on $\partial_\infty \mathbb{H}^k$ and the definition of
natural map $F$ associated to a non-elementary representation $\rho:\Gamma \rightarrow
\textup{Isom}(\mathbb{H}^m)$. The second section is devoted to the proof of the main
theorem. In the last section we describe some consequences of this result for the extendibility
of the volume function to the Morgan-Shalen compactification of
$X(\Gamma,\textup{Isom}(\mathbb{H}^3))$. We conclude by extending the main theorem to the more
general case of sequences of representations $\rho_n:\Gamma \rightarrow
\textup{Isom}(\mathbb{H}^m)$, where $\Gamma < \textup{Isom}^+(\mathbb{H}^k)$ such that
$\mathbb{H}^k/\Gamma$ is a complete hyperbolic $k$-manifold of finite volume and $m \geq k$.

\medskip

\textbf{Acknowledgements}: The authors would like to thank Juan Souto for the precious help and
the enlightening conversations and Thang Le for the essential information he gave us about the
evolution of this problem. We also thank the referee for his useful comments and remarks.


\section{Preliminary definitions}

\subsection{Barycentre of a measure}

We start by fixing some notation. From now until the end of the paper we are going to choose
the origin $O$ of the disk model as basepoint in $\mathbb{H}^k$. Moreover, we will use the same letter $O$ to denote basepoints in different hyperbolic spaces. Let $B_K(x,\theta)$ be the Busemann function of $\mathbb{H}^k$ normalized at $O$, that means for every $\theta \in \partial_\infty \mathbb{H}^k$ we set

\[
B_K(x,\theta)=\lim_{t \to \infty} d(x,c(t))-t,
\]
where $c$ is the geodesic ray starting at $O=c(0)$ and ending at $\theta$. The notation $B_K$
refers to the Busemann function relative to the $k$-dimensional hyperbolic space.

Let $\beta$ be a positive probability measure on $\partial_\infty \mathbb{H}^k$. Thanks to the convexity of Busemann functions the map

\[
\varphi_\beta: \mathbb{H}^k \rightarrow \R, \hspace{10pt} \varphi_\beta(y):=\int_{\partial_\infty \mathbb{H}^k} B_K(y,\theta)d\beta(\theta)
\]
is stricly convex, provided that $\beta$ is not the sum of two Dirac measures. Additionally, if the measure $\beta$ does not contain any atom of mass greater or equal than $1/2$, the following condition holds 

\[
\lim_{y \to \partial_\infty \mathbb{H}^k} \varphi_\beta(y)=\infty.
\]
and this implies that $\varphi_\beta$ admits a unique minimum in $\mathbb{H}^k$ (see~\cite[Appendix A]{besson95:articolo}).
On the other hand, if $\beta$ contains an atom of mass at least $1/2$, then it is readily
checked that the minimum of
$\varphi_\beta$ is $-\infty$ and it is attained at the atom. 

\begin{deft}
Let $\beta$ be any positive probability measure $\beta$ of finite mass which is not the sum of two Dirac masses with the same weight. If $\beta$ contains an atom $x$ of mass greater or equal than $1/2$ then we define its \textit{barycentre} as

\[
\textup{bar}_\mathcal{B}(\beta)=x,
\]
otherwise we define it as the point

\[
\textup{bar}_\mathcal{B}(\beta)=\textup{argmin}(\varphi_\beta).
\]

The letter $\mathcal{B}$ emphasizes the dependence of the construction on the Busemann functions. The barycentre of $\beta$ will be a point in $\overline{\mathbb{H}^k}$ which satisfies the following properties:

\begin{itemize}
	\item it is $\textup{weak-}^*$ continuous, that is if $\beta_n \to \beta$ in the
          $\textup{weak-}^*$ topology (and no measure is the sum of two atoms with equal weight)
           it holds
	
	\[
	\lim_{n \to \infty} \textup{bar}_\mathcal{B}(\beta_n)=\textup{bar}_\mathcal{B}(\beta)
	\]

	\item it is $\textup{Isom}(\mathbb{H}^k)$-equivariant, indeed for every $g \in
          \textup{Isom}(\mathbb{H}^k)$ (if $\beta$ is not the sum of two equal atoms) we have

	\[
	\textup{bar}_\mathcal{B}(g_*\beta)=g(\textup{bar}_\mathcal{B}(\beta)),
	\]
	
	\item when $\beta$ does not contain any atom of weight greater or equal than $1/2$, it is characterized by the following equation

          \begin{equation}\label{dadiff}
	\int_{\partial_\infty \mathbb{H}^k} dB_K|_{(\textup{bar}_\mathcal{B}(\beta),y)}(\cdot)d\beta(y)=0.
          \end{equation}

\end{itemize}

\end{deft}

\subsection{The Patterson-Sullivan family of measures and the BCG--natural map}

For more details about the following definitions and constructions we recomend the reader to see the first sections of~\cite{franc09:articolo}. Let $\Gamma < \textup{Isom}(\mathbb{H}^k)$ be a discrete group of divergence type, that is a subgroup for which the Poincar\'e series diverges at the critical exponent $\delta(\Gamma)$. For example, if $\Gamma$ is the fundamental group of a complete $k$-dimensional hyperbolic manifold $M$ of finite volume we have that $\delta(\Gamma)=k-1$. 

\begin{deft}
Let $\mathcal{M}^1(X)$ be the set of positive probability measures on a space $X$. The \textit{family of Patterson-Sullivan measures associated to $\Gamma$} is a family of measures \mbox{$\{\mu_x\} \in \mathcal{M}^1(\partial_\infty \mathbb{H}^k)$}, where $x \in \mathbb{H}^k$, which satisfies the following properties

\begin{itemize}
	\item the family is $\Gamma$-equivariant, that is $\mu_{\gamma x}=\gamma_*(\mu_x)$ for every $\gamma \in \Gamma$ and every $x \in \mathbb{H}^k$,
	\item For every $x,y \in \mathbb{H}^k$ it holds 

		\[
			d\mu_x(\theta)=e^{-\delta(\Gamma)B_y(x,\theta)}d\mu_y(\theta)
		\]
	where $B_y(x,\theta)$ is the Busemann function normalized at $y$.
\end{itemize}

\end{deft} 
 
If $\Gamma$ is the fundamental group of a complete $k$-dimensional hyperbolic manifold $M$ of
finite volume, let $\{ \mu_x \}$ be the family of Patterson-Sullivan measures associated to
$\Gamma$. We set $\mu=\mu_O$ and we notice that in the present case $\mu_O$ is the standard
visual measure on $\partial_\infty\mathbb H^k$ (i.e. the usual spherical Lebesgue measure).

 Let $\rho:\Gamma \rightarrow \textup{Isom}(\mathbb{H}^m)$ be a non-elementary
 representation. By both~\cite[Corollary 3.2]{burger3:articolo} and~\cite[Theorem 1.5]{franc09:articolo} there exists a
 $\rho$-equivariant measurable map $$D:\partial_\infty \mathbb{H}^k \rightarrow \partial_\infty
 \mathbb{H}^m$$ and two different maps of this type must agree on a full $\mu$-measure set.
We define

\[
\beta_x:=D_*(\mu_x).
\]

Cleary the measure $\beta_x$ lives in $\mathcal{M}^1(\partial_\infty \mathbb{H}^m)$ for every
$x$. We want to emphasize that starting from a point $x \in \mathbb{H}^k$ we end up with a
measure $\beta_x \in \mathcal{M}^1(\partial_\infty \mathbb{H}^m)$. 

Since we have a non-elementary representation, $\beta_x$ does not contain any atom of mass greater or equal than $1/2$. Indeed it holds

\begin{lem}
Let $\rho:\Gamma \rightarrow \textup{Isom}(\mathbb{H}^m)$ be a non-elementary representation and let $D:\partial_\infty \mathbb{H}^k \rightarrow \partial_\infty \mathbb{H}^m$ be a $\rho$-equivariant measurable map. Then $D(x)\neq D(y)$ for almost every $(x,y) \in \partial_\infty \mathbb{H}^k \times \partial_\infty \mathbb{H}^k$.
\end{lem}

\begin{proof}
Define the set $A:=\{ (x,y) \in \partial_\infty \mathbb{H}^k \times \partial_\infty
\mathbb{H}^k| D(x)=D(y)\}$. Since the map $D$ is $\rho$-equivariant, $A$ is a
$\Gamma$-invariant measurable subset of $\partial_\infty \mathbb{H}^k \times \partial_\infty
\mathbb{H}^k$. By the ergodicity of the action of $\Gamma$ on $\partial_\infty \mathbb{H}^k
\times \partial_\infty \mathbb{H}^k$ with respect to the measure $\mu\times\mu$ (see \cite{Yue96,Nic89,Rob00,Sul79}), the set $A$ must have
either null measure or full measure. By contradiction, suppose that $A$ has full measure. This
implies that for almost all $x$, the slice $A(x):=\{ y
\in \partial_\infty\mathbb{H}^k|D(x)=D(y)\}$ has full measure in $\partial_\infty
\mathbb{H}^k$. Isometries preserve the class of $\mu$, in particular, for any $\gamma\in
\Gamma$, if $A(x)$ has full measure then so does $\gamma A(x)$.
Since $\Gamma$ is countable, this implies that for almost all $x$, the set
$A_\Gamma(x):=\cap_{\gamma \in \Gamma} \gamma^{-1} A(x)$ has full measure in $\partial_\infty
\mathbb{H}^k$. Fix now a point $y \in A_\Gamma(x)$. For any $\gamma\in\Gamma$ we have
$(x,\gamma y)\in A$. In particular\footnote{We use $\gamma=id$ in the first equality and the last follows by equivariance of $D$.} 
 
\[
D(y)=D(x)=D(\gamma y)=\rho(\gamma)D(y)
\]

for every $\gamma \in \Gamma$, but this would imply that $\rho$ is elementary, which is a contradiction. 
\end{proof}

By the previous lemma, for all $x \in \mathbb{H}^k$, we can define 

\[
F(x):=\textup{bar}_\mathcal{B}(\beta_x)
\]
and this point will lie in $\mathbb{H}^m$. In this way we get a map $F:\mathbb{H}^k \rightarrow \mathbb{H}^m$. 

\begin{deft}
The map $F:\mathbb{H}^k \rightarrow \mathbb{H}^m$ is called \textit{natural map} for the
representation $\rho:\Gamma \rightarrow \textup{Isom}(\mathbb{H}^m)$.

Equation~(\ref{dadiff}) becomes

          \begin{equation}\label{dadiff2}
	\int_{\partial_\infty \mathbb{H}^m} dB_M|_{(F(x),y)}(\cdot)d\beta_x(y)=0.
          \end{equation}
and since $\beta_x=D_*(\mu_x)$, it can be rewritten as
          \begin{equation}\label{dadiff3}
	\int_{\partial_\infty \mathbb{H}^k} dB_M|_{(F(x),D(z))}(\cdot)d\mu_x(z)=0.
          \end{equation}

 The natural map is smooth and satisfies the following properties:

	\begin{itemize}
		\item Define the $k$-Jacobian of $F$ as

\[
Jac_k(F)(x):=\max_{u_1,\ldots,u_k}||D_xF(u_1) \wedge \ldots D_xF(u_k)||_{g_{\mathbb{H}^m}}
\]
where $\{ u_1,\ldots,u_k \}$ is an orthonormal frame of the tangent space $T_x\mathbb{H}^k$ with respect to the standard metric induced by $g_{\mathbb{H}^k}$ and the norm $||\cdot||_{g_{\mathbb{H}^m}}$ is the norm on $T_{F_n(x)}\mathbb{H}^m$ induced by $g_{\mathbb{H}^m}$. For every $k \geq 3$, we have $Jac_k(F)\leq 1$ and the equality holds at $x$ is and only is $D_xF:T_x\mathbb{H}^k \rightarrow T_{F(x)}\mathbb{H}^m$ is an isometry (see~\cite[Theorem 1.10]{besson99:articolo}).
  
	\item The map $F$ is $\rho$-equivariant, that is $F(\gamma x)=\rho(\gamma)F(x)$.
	\item By differentiating~(\ref{dadiff3}), one gets that for all $x \in \mathbb{H}^k$, $u
          \in T_x \mathbb{H}^k$, $v \in T_{F(x)}\mathbb{H}^m$ it holds 

	\begin{align*}
	&\int_{\partial_\infty \mathbb{H}^k} \nabla dB_M|_{(F(x),D(z))}(D_xF(u),v)d\mu_x(z)=\\
	&\delta(\Gamma) \int_{\partial_\infty \mathbb{H}^k} dB_M|_{(F(x),D(z))}(v)dB_K|_{(x,z)}(u)d\mu_x(z)
	\end{align*}
	where $\nabla$ is the Levi--Civita connection on $\mathbb{H}^m$. 
	\end{itemize}
\end{deft}

\begin{oss}
We need to require $k \geq 3$ to get the sharpness of the estimate on the Jacobian. Indeed,
this condition is equivalent to a necessary hypothesis which appears in~\cite[Lemma
B.4]{besson95:articolo}.~This point should become more explicit in Equations~$(\ref{eq:BCGAppB})$
and~$(\ref{eq:BCGAppBbis})$ at page~\pageref{eq:BCGAppB}.
\end{oss}

\subsection{Volume of representations and $\varepsilon$-natural maps}\label{epsilon}

If $\Gamma$ is the fundamental group of a non-compact, complete hyperbolic $k$-manifold $M$ of
finite volume, then $M$ is diffeomorphic to the interior of a compact manifold
$\overline M$ whose boundary consists of Euclidean $(k-1)$-manifolds. Denote each boundary
component by $T_i$ with $i=1,\ldots,h$. Recall that for each $T_i$ its fundamental group
$\pi_1(T_i) $ is an abelian parabolic subgroup of $\textup{Isom}(\mathbb{H}^k)$.

 Let $\rho:\Gamma \rightarrow \textup{Isom}(\mathbb{H}^m)$ be a representation and let $D:\mathbb{H}^k \rightarrow \mathbb{H}^m$ be a smooth $\rho$-equivariant map. We want to define its volume $\textup{Vol}(D)$. Let $g_{\mathbb{H}^m}$ be the standard hyperbolic metric on $\mathbb{H}^m$. The pullback of $g_{\mathbb{H}^m}$ along $D$ defines in a natural way a pseudo-metric on $\mathbb{H}^k$, which can be possibly degenerate, and hence it defines a natural $k$-form given by $\tilde \omega_D=\sqrt{|\det D^*g_{\mathbb{H}^m}|}$. The equivariance of $D$ with respect to $\rho$ implies that the form $\tilde \omega_D$ is $\Gamma$-invariant and hence it determines a $k$-form on $M$. Denote this form by $\omega_D$.

\begin{deft}
Let $\rho:\Gamma \rightarrow \textup{Isom}(\mathbb{H}^m)$ be a representation and let $D:\mathbb{H}^k \rightarrow \mathbb{H}^m$ be any smooth $\rho$-equivariant map. The \textit{volume} of $D$ is defined as

\[
\textup{Vol}(D):=\int_M \omega_D
\]
\end{deft}

\indent We keep denoting by $D:\mathbb{H}^k \rightarrow \mathbb{H}^m$ a generic smooth $\rho$-equivariant map. Since the fundamental group of each boundary component $T_i \subset  \partial \overline{M}$ is parabolic, it must fix a unique point on $\partial_\infty \mathbb{H}^k$. Define $c_i=\textup{Fix}(\pi_1T_i)$ and let $r(t)$ be a geodesic ray ending at $c_i$. We say that $D$ is a \textit{properly ending map} if all the limit points of $D(r(t))$ lie either in $\textup{Fix}(\rho(\pi_1T_i))$ or in a finite union of $\rho(\pi_1T_i)$-invariant geodesics.   

\begin{deft}
Given a representation $\rho:\Gamma \rightarrow \textup{Isom}(\mathbb{H}^m)$, we define its \textit{volume} as

\[
\textup{Vol}(\rho):=\inf \{ \textup{Vol}(D)| \text{ $D$ is smooth, $\rho$-equivariant and properly ending}\}.
\]
\end{deft}

When $\rho$ is non-elementary, a priori the BCG--natural map \mbox{$F:\mathbb{H}^k \rightarrow
  \mathbb{H}^m$} associated to $\rho$ is not a properly ending map, hence we cannot compare its
volume with the volume of representation $\rho$. However, for any $\varepsilon >0$ it is
possible to construct a family of smooth functions $F^\varepsilon: \mathbb{H}^k \rightarrow
\mathbb{H}^m$ that $C^1$-converge to $F$ as $\varepsilon \to 0$ and such that
$F^\varepsilon$ is a properly ending map for every $\varepsilon >0$ (see for instance~\cite[Lemma 4.5]{franc06:articolo}).
 
\begin{deft}
For any $\varepsilon >0$ there exists a map $F^\varepsilon:\mathbb{H}^k \rightarrow \mathbb{H}^m$ called $\varepsilon$-\textit{natural map} associated to $\rho$ which satisfies the following properties

\begin{itemize}
	\item $F^\varepsilon$ is smooth and $\rho$-equivariant,
	\item at every point of $\mathbb{H}^k$ we have $Jac_k(F^\varepsilon) \leq 1+\varepsilon$,
	\item for every $x \in \mathbb{H}^k$ it holds $\lim_{\varepsilon \to 0} F^\varepsilon(x)=F(x)$ and $\lim_{\varepsilon \to 0} D_xF^\varepsilon = D_xF$,
	\item $F^\varepsilon$ is a properly ending map.
\end{itemize}
\end{deft}

In particular, since $F^\varepsilon$ is a properly ending map, it holds trivially

\[
\textup{Vol}(\rho) \leq \int_M \sqrt{|\det((F^\varepsilon)^*g_{\mathbb{H}^m})|}.
\]

We are going to use the previous estimate later.

\section{Proof of Theorem~\ref{convergence}}

From now until the end of the section we are going to work in $\mathbb H^3$. We start by fixing the following setting. 

\begin{itemize}
 
	\item A group $\Gamma < \textup{Isom}(\mathbb{H}^3)$ so that $M=\mathbb{H}^3/\Gamma$ is
          a (non-compact) complete hyperbolic manifold of finite volume.
	\item A base-point $O \in \mathbb{H}^3$ used to normalize the Busemann function $B(x,\theta)$, with $x \in \mathbb{H}^3$ and $\theta \in \partial_\infty \mathbb{H}^3$.
	\item The family $\{ \mu_x \}$ of Patterson-Sullivan probability measures. Set $\mu=\mu_O$.
	\item A sequence of representations $\rho_n: \Gamma \rightarrow  \textup{Isom}(\mathbb{H}^3)$ such that $\lim_{n \to \infty}\textup{Vol}(\rho_n) = \textup{Vol}(M)$. 	
\end{itemize}

\begin{lem}
The condition $\lim_{n \to \infty}\textup{Vol}(\rho_n) = \textup{Vol}(M)$ implies that, up to pass to a subsequence, we can suppose that no $\rho_n$ is elementary.
\end{lem}

\begin{proof}
Elementary representations have zero volume and $\lim_{n \to \infty}\textup{Vol}(\rho_n) = \textup{Vol}(M)$, which is stricly positive. 
\end{proof}

With an abuse of notation we still denote the subsequence of the previous lemma by $\rho_n$. Since no $\rho_n$ is elementary we can consider the sequence of $\rho_n$-equivariant measurable maps $D_n: \partial_\infty \mathbb{H}^3 \rightarrow \partial_\infty \mathbb{H}^3$ and the corresponding sequence of BCG--natural maps $F_n:\mathbb{H}^3 \rightarrow \mathbb{H}^3$.

\begin{lem}
Up to conjugating $\rho_n$ by a suitable element $g_n \in \textup{Isom}(\mathbb{H}^3)$, we can suppose $F_n(O)=O$. 
\end{lem}

\begin{proof}
Conjugating $\rho_n$ by $g$ reflects in post-composing $F_n$ with $g$. We can choose $g_n$ such $g_n(F_n(O))=O$. 
\end{proof}

The choice to fix the origin of $\mathbb{H}^3$ as the image of $F_n(O)$ is made to avoid pathological behaviour. For instance consider a sequence of loxodromic elements $g_n \in \textup{Isom}(\mathbb{H}^3)$ which is divergent and define the representations $\rho_n:=g_n \circ i \circ g_n^{-1}$, where $i:\Gamma \rightarrow \textup{Isom}(\mathbb{H}^3)$ is the standard lattice embedding. Clearly this sequence of representations satisfies $\lim_{n \to \infty} \textup{Vol}(\rho_n)=\textup{Vol}(M)$ since for every $n \in \N$ we have $\textup{Vol}(\rho_n)=\textup{Vol}(M)$. However, there does not exist any subsequence of $\rho_n$ converging to the holonomy of the manifold $M$. 

\begin{deft}
For any $n \in \N$ and every $x \in \mathbb{H}^3$ we can define self-adjoint operators $K_n$ and $H_n$
on $T_{F_n(x)}\mathbb{H}^3$ via the following implicit formulas:

\[
\langle K_n|_{F_n(x)} u, u \rangle= \int_{\partial_\infty \mathbb{H}^3} \nabla dB|_{(F_n(x),D_n(\theta))}(u,u)d\mu_x(\theta)
\]

\[
\langle H_n|_{F_n(x)} u, u \rangle= \int_{\partial_\infty \mathbb{H}^3} (dB|_{(F_n(x),D_n(\theta))}(u))^2d\mu_x(\theta)
\]
for any $u \in T_{F_n(x)}\mathbb{H}^3$. The notation $\langle \cdot , \cdot \rangle$ stands for the scalar product on $T_{F_n(x)}\mathbb{H}^3$ induced by the hyperbolic metric on $\mathbb{H}^3$ . 
\end{deft}

For sake of simplicity we are going to drop the subscripts in $K_n$ and $H_n$. Recall that, since both the domain and the target have the same dimension, the $3$-jacobian $Jac_3(F_n)$ coincides the modulus of the jacobian determinant $\det(D_xF_n)$. As stated in~\cite[Lemma 5.4]{besson96:articolo}, the following inequality holds for every $x \in \mathbb{H}^3$

\[
|\det(D_xF_n)| \leq \left(\frac{4}{3}\right)^\frac{3}{2}\frac{\det(H_n)^\frac{1}{2}}{\det(K_n)}.
\]

\begin{lem}\label{almost_everywhere}
Suppose $\lim_{n \to \infty} \textup{Vol}(\rho_n)=\textup{Vol}(M)$. Hence we have that $|\det(D_xF_n)|$ converges to $1$ almost everywhere in $\mathbb{H}^3$ with respect to the measure induced by the standard metric.
\end{lem}

\begin{proof}
Denote by $F_n^\varepsilon: \mathbb{H}^3 \rightarrow \mathbb{H}^3$ the $\varepsilon$-natural maps introduced in Section~\ref{epsilon}. Recall that we have the following estimate

\[
\textup{Vol}(\rho_n) \leq \int_M |\det(D_xF_n^\varepsilon)|d \textup{vol}_{\mathbb{H}^3}(x)=\textup{Vol}(F_n^\varepsilon)
\]
and since $|\det(D_xF^\varepsilon_n)| \leq 1+\varepsilon$ and $\lim_{\varepsilon \to 0} D_xF^\varepsilon_n=D_xF_n$, by the theorem of dominated convergence we get

\[
\textup{Vol}(\rho_n) \leq \int_M |\det(D_xF_n)| d\textup{vol}_{\mathbb{H}^3}(x) \leq \textup{Vol}(M)
\]
from which follows the statement. 
\end{proof}

If $\mathcal{N}$ is the set of zero measure outside of which $|\det(D_xF_n)|$ is converging, for every $x \in \mathbb{H}^3 \setminus \mathcal{N}$ and fixed $\varepsilon>0$ there must exist $n_0=n_0(\varepsilon,x)$ such that $|\det(D_xF_n)| \geq 1-\varepsilon$ for every $n > n_0$. Thus it holds

\[
\left(\frac{4}{3}\right)^\frac{3}{2}\frac{\det(H_n)^\frac{1}{2}}{\det(K_n)} > 1-\varepsilon
\]
from which we can deduce

\[
\frac{\det(H_n)}{(\det(K_n))^2} > \left(\frac{3}{4}\right)^3(1-\varepsilon)^2 > \left(\frac{3}{4}\right)^3(1-2\varepsilon).
\]

Moreover, since $\mathbb{H}^3$ has costant sectional curvature equal to $-1$, we have $K_n=I-H_n$ (see~\cite[Section 5.b]{besson95:articolo}). Here $I$ stands for the identity on $T_{F_n(x)}\mathbb{H}^3$. Hence, by substituting the expression of $K_n$ in the previous inequality, we get 

\[
\frac{\det(H_n)}{(\det(I-H_n))^2} > \left(\frac{3}{4}\right)^3(1-2\varepsilon).
\]

Consider now the set of positive definite symmetric matrices of order $3$ with real entries and trace equal to $1$, namely

\[
Sym^+_1(3,\R):=\{ H \in Sym(3,\R)| H>0, \textup{Tr}(H)=1\}.
\]

Once we have fixed a basis of $T_{F_n(x)}\mathbb{H}^3$, we can identify $H_n$ and $K_n$ with the matrices representing these bilinear forms with respect to the fixed basis. Under this assumption, recall that $H_n \in Sym^+_1(3,\R)$ for every $n \in \N$, as shown in~\cite[Proposition B.1]{besson96:articolo}. If we define 

\begin{equation}
  \label{eq:BCGAppB}
\psi:Sym^+_1(3,\R) \rightarrow \R, \hspace{10pt} \psi(H)=\frac{\det(H)}{(\det(I-H))^2},  
\end{equation}
we know that 
\begin{equation}
  \label{eq:BCGAppBbis}
\psi(H) \leq \left(\frac{3}{4}\right)^3  
\end{equation}

and the equality holds if and only if $H=I/3$ (see~\cite[Appendix B]{besson95:articolo}).
\begin{oss}\label{rem:BCGAppB}
Note that if $k=2$, then $\psi(H)$ is unbounded on the space of symmetric positive definite
matrices with trace equal to $1$. This is the reason why BCG method fails (as expected) in the case of surfaces.  
\end{oss}

It is worth noticing that the space $Sym^+_1(3,\R)$ is not compact and a priori there could exist a sequence of elements $H_n \in Sym^+_1(3,\R)$ such that $$\lim_{n\to \infty} \psi(H_n)=\left(\frac{3}{4}\right)^3.$$

We are going to show that this is impossible.

\begin{prop}\label{maximum}
Suppose to have a sequence $H_n \in Sym^+_1(3,\R)$ such that 
\[
\lim_{n\to \infty} \psi(H_n)=\left(\frac{3}{4}\right)^3.
\]

Hence the sequence $H_n$ must converge to $I/3$. 
\end{prop}

\begin{proof}
We start by observing that the function $\psi$ is invariant by conjugation for an element $g \in GL(3,\R)$. Indeed, $\psi(H)$ can be expressed as $\psi(H)=p_H(0)/(p_H(1))^2$, where $p_H$ is the characteristic polynomial of $H$. Hence the claim follows. In particular, we have an induced function 

\[
\tilde \psi : O(3,\R) \backslash Sym^+_1(3,\R) \rightarrow \R, \hspace{10pt} \tilde \psi(\bar H)=\psi(H),
\]
where $\bar H$ denotes the equivalence class of the matrix $H$ and the orthogonal group $O(3,\R)$ acts on $Sym^+_1(3,\R)$ by conjugation. We can think of the space $O(3,\R) \backslash Sym^+_1(3,\R)$ as the interior $\mathring \DD_2$ of the standard 2-simplex quotiented by the action of the symmetric group $\mathfrak{S}_3$ which permutes the coordinate of an element $(\lambda_1,\lambda_2,\lambda_3) \in \mathring \DD_2$. An explicit homeomorphism between the two spaces is given by

\[
\Lambda: O(3,\R) \backslash Sym^+_1(3,\R) \rightarrow \mathfrak{S}_3 \backslash \mathring \DD_2, \hspace{10pt} \Lambda(\bar H):=[\lambda_1(H),\lambda_2(H),\lambda_3(H)],
\] 
where $\lambda_i(H)$ for $i=1,2,3$ are the eigenvalues of $H$. By defining $\Psi=\psi \circ \Lambda^{-1}$, we can express this function as

\[
\Psi: \mathfrak{S}_3 \backslash \mathring \DD_2 \rightarrow \R, \hspace{10pt} \Psi([a,b,c])=\frac{abc}{((1-a)(1-b)(1-c))^2}. 
\]

We are going to think of $\Psi$ as defined on $\mathring \DD_2$ and we are going to estimate this function on the boundary of $\DD_2$. Since $a+b+c=1$, with an abuse of notation we will write 

\[
\Psi(a,b)=\frac{ab(1-a-b)}{((1-a)(1-b)(a+b))^2}. 
\]
identifying $\mathring \DD_2$ with the interior of the triangle $\tau$  in $\R^2$ with vertices $(0,0)$, $(1,0)$ and $(0,1)$. If a sequence of points is converging to a boundary point of $\DD_2$, then we have a sequence $(a_n,b_n)$ of points converging to a boundary point of $\tau$. If the limit point is not a vertex of $\tau$ then $\lim_{n \to \infty} \Psi(a_n,b_n) = 0$. For instance, suppose $\lim_{n \to \infty} (a_n,b_n)=(\alpha,0)$ with $\alpha \neq 0,1$. Hence 

\[
\lim_{n \to \infty} \Psi(a_n,b_n)=\lim_{n \to \infty} \frac{a_nb_n(1-a_n-b_n)}{((1-a_n)(1-b_n)(a_n+b_n))^2}=0
\]  
as claimed. For the other boundary points which are not vertices, the computation is the same. The delicate points are given by the vertices $(0,0)$, $(1,0)$ and $(0,1)$. On these points the function $\Psi$ cannot be continuously extended. However we can uniformly bound the possible limit values. Suppose to have a sequence $(a_n,b_n)$ such that $\lim_{n \to \infty} (a_n,b_n)=(0,0)$. We have

\[
\Psi(a_n,b_n)=\frac{a_nb_n(1-a_n-b_n)}{((1-a_n)(1-b_n)(a_n+b_n))^2} \sim \frac{a_nb_n}{(a_n+b_n)^2} \leq \frac{1}{4}.
\]  
where the symbol $\sim$ denotes that the sequence on the left has the same behaviour of the sequence of the right in a neighborhood of $(0,0)$. Analogously, if $\lim_{n \to \infty} (a_n,b_n)=(1,0)$ then

\[
\Psi(a_n,b_n)=\frac{a_nb_n(1-a_n-b_n)}{(1-a_n)(1-b_n)(a_n+b_n)}\sim \frac{b_n}{1-a_n}\left(1-\frac{b_n}{1-a_n}\right) \leq \frac{1}{4}.
\]
and the same for $\lim_{n \to \infty} (a_n,b_n)=(0,1)$. The previous computation proves that $\Psi$ is uniformly bounded by $1/4$ on the boundary of $\tau$, hence on the boundary of $\DD_2$. Equivalently $\psi$ is bounded by $1/4$ in a suitable neighborhood at infinity of $Sym^+_1(3,\R)$, from which follows the statement.
\end{proof}

We know that in our context we have

\[
\left(\frac{3}{4}\right)^3(1-2\varepsilon) \leq \psi(H_n) \leq \left(\frac{3}{4}\right)^3
\]
for $n \geq n_0$. As a consequence of Proposition~\ref{maximum}, the sequence $H_n$ must converge to $I/3$. Hence $H_n$ converges to $I/3$ almost-everywhere on $\mathbb{H}^3$. We are going to prove that this implies the uniform convergence of $H_n$ to $I/3$ on compact sets. Before doing this we recall these two lemmas which can be found in~\cite[Section 7]{besson95:articolo}.

\begin{lem}\label{kappa1}
Let $x,x' \in \mathbb{H}^3$ such that the maximum eigenvalue of $H_n$ satisfies $\lambda_n \leq 2/3$ at every point of the geodesic joining $x$ to $x'$. Then there exists a positive constant $C_1$ such that

\[
d(F_n(x),F_n(x')) \leq C_1d(x,x').
\]
\end{lem}

\begin{lem}\label{kappa2}
Let $x,x' \in \mathbb{H}^3$. Let $P$ be the parallel transport from $F_n(x)$ to $F_n(x')$ along the geodesic which joins the two points. Denote by $H_n(x)$ the endomorphism defined on $T_{F_n(x)}\mathbb{H}^3$. Then there exists a positive constant $C_2$ such that

\[
||H_n(x) - H_n(x') \circ P|| \leq C_2(d(x,x')+d(F_n(x),F_n(x'))).
\]

The norm which appears above is the one obtained by thinking of each endomorphism as an operator between Euclidean vector spaces. 
\end{lem}

\begin{prop}
Suppose the sequence $H_n$ converges almost everywhere to $I/3$. Thus it converges uniformly to $I/3$ on every compact set of $\mathbb{H}^3$. 
\end{prop}

\begin{proof}
We will follow the same proof of~\cite[Lemma 7.5]{besson95:articolo}. Without loss of generality we may reduce ourselves to the case of a closed ball $\overline{B_r(O)}$ around the origin of the Poincar\'e model of $\mathbb{H}^3$. Since $H_n$ is converging almost everywhere to $I/3$ on $\mathbb{H}^3$, hence in particular on $\overline{B_r(O)}$, by Egorov theorem, given a fixed $\eta >0$ there will exist a compact set $K$ and $N \in \N$ such that $\textup{Vol}(\overline{B_r(O)} \setminus K) < \eta$ and 

\[
||H_n(x) - I/3||< \epsilon
\]
for every $n \geq N$ and every $x \in K$. Moreover we can assume that the set $\overline{B_r(O)} \setminus K$ is sufficiently small not to contain any ball of radius $\epsilon$, for $\epsilon >0$.  This assumption implies that for every $x \in \overline{B_r(O)}$ we must have $d(x,K)<\epsilon$. Fix now $\epsilon$, $K$ and a suitable value $n \geq N$ so that 

\[
||H_n(x) - I/3|| < \epsilon 
\]
for every $x \in K$. As in Lemma~\ref{kappa2} we will write $H_n(x)$ to denote the endomorphism $H_n$ defined on $T_{F_n(x)}\mathbb{H}^3$. By contradiction, suppose the statement is false. There must exist two points $x'_n \in \overline{B_r(O)}$ and $x_n \in K$ so that $d(x_n,x_n')<\epsilon$ and $||H_n(x_n')-I/3||>C_3 \epsilon$, where we can assume

\[
\frac{1}{3 \epsilon} \geq C_3 \geq C_2(C_1+1)+1
\]
and $C_1$ and $C_2$ are the constants introduced in the previous lemmas.\\
\indent The continuity of the function $x \to H_n(x)$ implies the existence of a point $x''_n$ contained in the geodesic segment $[x_n,x_n']$ such that $||H_n(x_n'')-I/3||=C_3 \epsilon$. This implies that the maximum eigenvalue of $H_n$ satisfies $\lambda_n \leq 2/3$ at every point of the geodesic segment $[x_n,x_n'']$. By applying Lemma~\ref{kappa1} and Lemma~\ref{kappa2} we get that 

\[
||H_n(x_n) - H_n(x_n'')\circ P|| \leq C_2(C_1+1)\epsilon,
\]
where $P$ is the parallel transport from $F_n(x_n)$ to $F_n(x_n'')$ along the geodesic segment joining them. Since $||H_n(x_n)-I/3|| < \epsilon$ we get a contradiction.
\end{proof}

Thus, if we consider a closed ball $\overline{B_r(O)}$ with $r>0$, there exists $n_1=n_1(\varepsilon,r)$ such that for $n>n_1$ we have the following estimates

\[
|2/3 \langle  D_xF_n (v),u \rangle| - \varepsilon <  |\langle K_n \circ D_xF_n(v),u \rangle|, \hspace{10pt} \langle H_n u , u \rangle^\frac{1}{2} < ||u||/\sqrt{3} + \varepsilon.
\]

As a consequence of the Cauchy--Schwarz inequality, we can write

\[
| \langle K_n \circ D_xF_n (v),u \rangle| \leq 2 (\langle H_n(u),u \rangle)^{\frac{1}{2}}(\int_{\partial_\infty \mathbb{H}^3} (dB|_{(x,\theta)} (v))^2 d\mu_x(\theta))^\frac{1}{2},
\]
for every $v \in T_x\mathbb{H}^3$ and $u \in T_{F_n(x)}\mathbb{H}^3$. Hence by taking $n>n_1$ we get

\[
|2/3 \langle D_xF_n(v),u \rangle| - \varepsilon  \leq 2(||u||/\sqrt{3} + \varepsilon ) (\int_{\partial_\infty \mathbb{H}^3} (dB|_{(x,\theta)} (v))^2 d\mu_x(\theta))^\frac{1}{2}.
\]

Recall that $||dB||^2=1$. By considering on both sides the supremum on all the vectors $u$ of norm equal to $1$ we get

\[
||D_xF_n(v)|| < \sqrt{3} ||v|| + 3\varepsilon(||v||+1/2)
\]

Again, by taking the supremum on all the vectors $||v||=1$ we get

\[
||D_xF_n|| < \sqrt{3} + 9/2\varepsilon
\]
hence $||D_xF_n||$ is uniformly bounded on $\overline{B_r(O)}$ for any $n>n_1$ and for any choice of $r>0$. We are now ready to prove Theorem~\ref{convergence}.

\begin{proof}
Since we know that $\lim_{n \to \infty} \textup{Vol}(\rho_n)=\textup{Vol}(M)$, the previous computations shows that $||D_xF_n||$ must be eventually uniformly bounded on every compact set of $\mathbb{H}^3$. Let $x \in \mathbb{H}^3$ be any point and let $\gamma \in \Gamma$. Let $c$  be the geodesic joining $x$ to $\gamma x$. Denote by $L=d(x,\gamma x)$ so that the interval $[0,L]$ parametrizes the curve $c$. Consider a closed ball $\overline{B_r(O)}$ sufficiently large to contain in its interior both $x$ and $\gamma x$. On this ball there must exist a constant $C$ such that $||D_xF_n|| < C$ for $n$ bigger than a suitable value $n_0$. Thus, it holds

\[
d(F_n(x),F_n(\gamma x)) \leq \int_0^L ||D_{c(t)}F_n(\dot c(t))|| dt \leq \int_0^L ||D_{c(t)}F_n||dt \leq Cd(x,\gamma x).
\]

Recall that given an element $g \in \textup{Isom}(\mathbb{H}^3)$ its translation length is defined as $\mathfrak{L}_{\mathbb{H}^3}(g):=\inf_{y \in \mathbb{H}^3}d(gy,y)$. The previous estimate implies that the translation length of the element $\rho_n(\gamma)$ can be bounded by

\[
\mathfrak{L}_{\mathbb{H}^3}(\rho_n(\gamma)) \leq d(\rho_n(\gamma)F_n(x),F_n(x)) \leq Cd(\gamma x,x)
\]
and hence the sequence $\rho_n$ is bounded in the character variety $X(\Gamma,\textup{Isom}(\mathbb{H}^3))$. Moreover the choice made before to fix $F_n(O)=O$ guarantees that the sequence $\rho_n$ must converge to a representation $\rho_\infty$. By the continuity of the volume with respect to the pointwise convergence, we get 

\[
\textup{Vol}(\rho_\infty)=\lim_{n \to \infty} \textup{Vol}(\rho_n)=\textup{Vol}(M).
\]

By the rigidity of volume function we know that $\rho_\infty$ must be conjugated to $i$, and the theorem is proved. 
\end{proof}

\section{Consequences and generalizations of Theorem~\ref{convergence}}

In this section we are going to prove Corollary~\ref{rigidity} and state a consequence
regarding the Morgan--Shalen compactification of $X(\Gamma,\textup{Isom}(\mathbb{H}^3))$. We
also discuss generalizations of Theorem~\ref{convergence} to higher dimensional cases.
We begin with the proof of Corollary~\ref{rigidity}. 

\begin{proof}[Proof of Corollary~\ref{rigidity}.]
If there did not exist such an $\varepsilon$, we should have $\textup{Vol}(\rho_n) \rightarrow \textup{Vol}(M)$, but this contraddicts Theorem~\ref{convergence}. Indeed the sequence $\rho_n$ should converge to a representation conjugated to the standard lattice embedding $i:\Gamma \rightarrow \textup{Isom}(\mathbb{H}^3)$ and it could not converge to an ideal point.
\end{proof}

The previous result has a clear consequence in the study of the volume function on the character variety $X(\Gamma)=X(\Gamma,\textup{Isom}(\mathbb{H}^3))$. Let $\overline{X(\Gamma)}^{MS}$ be the Morgan--Shalen compactification of the character variety $X$ (see~\cite{morgan2:articolo} for a definition). The previous corollary can be restated as follows

\begin{cor}
Let $\textup{Vol}:X(\Gamma) \rightarrow \R$ be the volume function. Let $\mathcal{N}(i)$ be a small neighborhood in $X(\Gamma)$ of the class containing the standard lattice embedding $i$ with respect to the topology of the pointwise convergence. Suppose that there exists a continuous extension $\overline{\textup{Vol}}:\overline{X(\Gamma)}^{MS} \rightarrow \R$. Hence we can bound uniformly the restriction

\[
\overline{\textup{Vol}}|_{\overline{X(\Gamma)}^{MS} \setminus \mathcal{N}(i)} < \textup{Vol}(M) - \varepsilon
\]
with a suitable value of $\varepsilon >0$. 
\end{cor}

In particular, the previous corollary proves~\cite[Conjecture 1]{guilloux:articolo} and
hence~\cite[Theorem 1.2]{guilloux:articolo} for representations into $PSL(2,\C)$.

\bigskip

Now we prove now a generalization of Theorem~\ref{convergence} when $M$ is a $k$-manifold and
$\rho_n$ takes values in $\textup{Isom}(\mathbb H^k)$ (for $k>3$).

More precisely, let $\Gamma$ be the fundamental group of a complete hyperbolic $k$-dimensional
manifold $M$ with finite volume. We show that, given a sequence of representations
$\rho_n:\Gamma \rightarrow \textup{Isom}(\mathbb{H}^k)$ such that $\lim_{n \to \infty}
\textup{Vol}(\rho_n)=\textup{Vol}(M)$, it is possible to find a sequence of elements $g_n \in
\textup{Isom}(\mathbb{H}^k)$ such that the sequence $g_n \circ \rho_n \circ g_n^{-1}$ converges
to the standard lattice embedding $i:\Gamma \rightarrow \textup{Isom}(\mathbb{H}^k)$. The key
point of the proof in the case $k=3$ is given by Proposition~\ref{maximum}, which is still valid in dimension bigger or equal than $4$. Indeed, following what we have done before, consider the space

\[
Sym^+_1(k,\R):=\{ H \in Sym(k,\R)| H>0, \textup{Tr}(H)=1\}
\]
of real symmetric matrices of order $k$ with trace equal to $1$ which are positive definite. The function 

\[
\psi:Sym^+_1(k,\R) \rightarrow \R, \hspace{10pt} \psi(H):=\frac{\det(H)}{(\det(H-I))^2}
\]
induces a function $\tilde \psi$ on the quotient $O(k,\R) \backslash Sym^+_1(k,\R)$, where the orthogonal group acts by conjugation. As in the case of $k=3$, denote by $\mathring \DD_{k-1}$ the interior of the standard $(k-1)$-simplex and consider the action of $\mathfrak{S}_k$ by permutation of coordinates. We can read the function $\psi$ on the space $\mathfrak{S}_k \backslash \mathring \DD_{k-1}$ by considering

\[
\Psi(a_1,\ldots,a_k):=\prod_{i=1}^k \frac{a_i}{(1-a_i)^2}.
\]

Indeed the space $O(k,\R) \backslash Sym^+_1(k,\R)$ is homeomorphic to the space $\mathfrak{S}_k \backslash \mathring \DD_{k-1}$ and the homeomorphism is realized by sending the class of a symmetric matrix $H$ to the non-ordered $n$-tuple of its eigenvalues. We are now interest in extending the function $\Psi$ to the space $\mathfrak{S}_k \backslash \DD_{k-1}$ and to do this we are going to consider the function as defined on $\mathring \DD_{k-1}$. Moreover, since $\sum_{i=1}^k a_i=1$, with an abuse of notation, we are going to rewrite $\Psi$ as

\[
\Psi(a_1,\ldots,a_{k-1})=\frac{a_1 \ldots a_{k-1}(1-\sum_{i=1}^{k-1}a_i)}{(1-a_1)^2\ldots(1-a_{k-1})^2(\sum_{i=1}^{k-1}a_i)^2}.
\]

On every point of the boundary $\partial \DD_{k-1}$ which is not a vertex, the function clearly extends with zero. The same holds for the vertex $(1,0,\ldots,0)$ corresponding to the $(k-1)$-tuple $(0,0,\ldots,0)$. Indeed, near $(0,0,\ldots,0)$ we have

\[
\Psi(a_1,\ldots,a_{k-1}) \sim \frac{a_1\ldots a_{k-1}}{(\sum_{i=1}^{k-1} a_i)^2} \leq \frac{(\sum_{i=1}^{k-1} a_i)^{k-3}}{(k-1)^{k-1}}
\]
where the symbol $\sim$ denotes that $\Psi$ has the same behaviour of the expression on the right. For $k \geq 4$ the right-hand side is a function which converges to zero as $(a_1,\ldots,a_{k-1}) \to (0,\ldots,0)$. Moreover, since the function $\Psi$ is invariant under the action of $\mathfrak{S}_k$ on $\mathring \DD_{k-1}$ we have that its continuous extension must satisfy

\[
\Psi(1,0,\ldots,0)=\Psi(0,1,\ldots,0)=\Psi(0,0,\ldots,1)
\]
and so the function can be extended to zero at any vertex. In particular given a sequence of matrices $H_n$ such that $\lim_{n \to \infty} \psi(H_n)=(k/(k-1)^2)^k$ we have that the sequence $H_n$ must converge to $I/k$, where $I$ is the identity matrix of order $k$. From the previous considerations and following the same strategy of the case $k=3$, it is straightforward to prove

\begin{teor}\label{convk=m}
Let $\Gamma$ be the fundamental group of a complete hyperbolic $k$-dimensional non-compact manifold of finite volume. Let $\rho_n:\Gamma \rightarrow \textup{Isom}(\mathbb{H}^k)$ be a sequence of representations such that $\lim_{n \to \infty} \textup{Vol}(\rho_n)=\textup{Vol}(M)$. It is possible to find a sequence of elements $g_n \in \textup{Isom}(\mathbb{H}^k)$ such that the sequence $g_n \circ \rho_n \circ g_n^{-1}$ converges to the standard lattice embedding $i:\Gamma \rightarrow \textup{Isom}(\mathbb{H}^k)$.
\end{teor}

From which we deduce

\begin{cor}
Suppose $\rho_n:\Gamma  \rightarrow \textup{Isom}(\mathbb{H}^k)$ is a sequence of representations converging to any ideal point of the Morgan--Shalen compactification of $X(\Gamma,\textup{Isom}(\mathbb{H}^k))$. Then the sequence of volumes $\textup{Vol}(\rho_n)$ must be bounded from above by $\textup{Vol}(M)-\varepsilon$ with $\varepsilon>0$. 
\end{cor}

\begin{oss}
As mentioned in the introduction, it is worth noticing that if a sequence of representations $\rho_n:\Gamma \rightarrow \textup{Isom}(\mathbb{H}^k)$ satisfies $\lim_{n \to \infty} \textup{Vol}(\rho_n)=\textup{Vol}(M)$, with $k \geq 4$, then Theorem~\ref{convk=m} and~\cite[Theorem 2.3]{garland:articolo} imply that it must be eventually constant in the character variety $X(\Gamma,\textup{Isom}(\mathbb{H}^k))$. 
\end{oss}

\bigskip

We conclude by discussing the generalization of Theorem~\ref{convergence} to the case where $M$ is a
$k$-manifold  and $\rho_n$ takes values in $\textup{Isom}(\mathbb H^m)$ (with $m>k\geq 3$).

More precisely, let $\rho_n:\Gamma \rightarrow \textup{Isom}(\mathbb{H}^m)$ be
a sequence of representations such that $\lim_{n \to \infty}
\textup{Vol}(\rho_n)=\textup{Vol}(M)$. We show that there exists a sequence $g_n \in
\textup{Isom}(\mathbb{H}^m)$ such that the sequence $g_n \circ \rho_n \circ g_n^{-1}$ converges
to a representation $\rho_\infty$ which preserves a totally geodesic copy of $\mathbb{H}^k$ and
whose $\mathbb{H}^k$-component is conjugated to the standard lattice embedding $i:\Gamma
\rightarrow \textup{Isom}(\mathbb{H}^k) < \textup{Isom}(\mathbb{H}^m)$.

The proof in this general case follows the line of the case $k=m$ but it needs some additional
care. We do not rewrite the whole proof but we only concentrate on the subtleties which differ 
from the previous case. 

Let $F_n:\mathbb{H}^k \rightarrow \mathbb{H}^m$ be the natural map associated to the
representation $\rho_n$. We are going to follow~\cite{besson99:articolo} for the
notation. Recall that $B_M$ denotes the Busemann function relative to the hyperbolic space of
dimension $m$ centered at the origin $O$. Similarly to what we have done before, for any $n \in
\N$ and every $x \in \mathbb{H}^k$ we (implicitly) define self-adjoint operators $K_n,H_n$ on
$T_{F_n(x)}\mathbb{H}^m$ by:

\[
\langle K_n|_{F_n(x)} u, u \rangle= \int_{\partial_\infty \mathbb{H}^m} \nabla dB_M|_{(F_n(x),D_n(\theta))}(u,u)d\mu_x(\theta)
\]

\[
\langle H_n|_{F_n(x)} u, u \rangle= \int_{\partial_\infty \mathbb{H}^m} (dB_M|_{(F_n(x),D_n(\theta))}(u))^2d\mu_x(\theta)
\]
for any $u \in T_{F_n(x)}\mathbb{H}^m$. Since the dimension $m$ is bigger than $k$, we will
need to define another operator $H'$, this time on $T_x\mathbb{H}^k$. For any $v \in T_x\mathbb{H}^k$, we set

\[
\langle H'_n|_x v, v \rangle= \int_{\partial_\infty \mathbb{H}^k} (dB_K|_{(x,\theta)}(v))^2d\mu_x(\theta).
\]

For simplicity, we are going to drop the subscript which refers to the tangent space on which operators are defined. As a consequence of the Cauchy--Schwarz inequality we get

\[
\langle K_n \circ D_xF_n(v),u \rangle \leq (k-1) (\langle H_n (u),u \rangle)^{\frac{1}{2}} (\langle H'_n (v),v \rangle)^{\frac{1}{2}}
\]
for every $v \in T_x\mathbb{H}^k$ and every $u \in T_{F_n(x)}\mathbb{H}^m$.\\
\indent By applying the same strategy of the proof of Lemma~\ref{almost_everywhere}, we get that the condition $\lim_{n \to \infty}\textup{Vol}(\rho_n)=\textup{Vol}(M)$ implies that the $k$-Jacobian $Jac_k(F_n)$ of the natural maps $F_n$ converges to $1$ almost everywhere with respect the measure induced by the standard hyperbolic metric on $\mathbb{H}^k$. Since 

\[
Jac_k(F_n)(x):=\max_{u_1,\ldots,u_k}||D_xF_n(u_1) \wedge \ldots D_xF_n(u_k)||_{g_{\mathbb{H}^m}}, 
\]
let $\{ u_1,\ldots,u_k \}$ be the frame which realizes the maximum and denote by $U_x$ the subspace $U_x:=\textup{span}_\R\{u_1,\ldots,u_k\}$ of $T_x\mathbb{H}^k$ (in fact the subspace $U_x$ coincides with $T_x\mathbb{H}^k$, but we prefer to mantain the same notation of~\cite{besson99:articolo}). Set $V_{F_n(x)}:=D_xF_n(U_x)$. We denote by $K_n(x)^V$, $H_n(x)^V$ and $H'_n(x)^U$ the restrictions of the form $K_n|_{F_n(x)}$, $H_n|_{F_n(x)}$ and $H_n'|_x$ to the  subspace $V_{F_n(x)}$, $V_{F_n(x)}$ and $U_x$, respectively. As consequence of the Cauchy--Schwarz inequality, as in~\cite[Section 2]{besson99:articolo} it results

\begin{eqnarray*}
&~&\det(K_n(x)^V)Jac_k(F_n)(x) \\&\leq&  (k-1)^k (\det (H^V_n(x)))^\frac{1}{2} (\det (H'^U_n(x)))^\frac{1}{2} \\&\leq&  k^{-\frac{k}{2}}(k-1)^k(\det (H^V_n(x)))^\frac{1}{2} 
\end{eqnarray*}

and since $K^V_n(x)=I - H^V_n(x)$ we get the estimate

\[
Jac_k F_n(x) \leq \frac{(k-1)^k}{k^\frac{k}{2}} \frac{(\det(H^V_n(x)))^\frac{1}{2}}{\det(I - H^V_n(x))}.
\]

In this way we can apply the same strategy followed for the case $k=m$ and hence it is straightforward to prove

\begin{teor}\label{convk>3}
Let $\Gamma$ be the fundamental group of a complete hyperbolic $k$-dimensional non-compact
manifold of finite volume, with $k \geq 3$. Consider an integer $m \geq k$. Given a sequence of representations
$\rho_n:\Gamma \rightarrow \textup{Isom}(\mathbb{H}^m)$ such that $\lim_{n \to \infty}
\textup{Vol}(\rho_n)=\textup{Vol}(M)$, there exists a sequence of elements $g_n \in
\textup{Isom}(\mathbb{H}^m)$ such that the sequence $g_n \circ \rho_n \circ g_n^{-1}$ converges
to a representation $\rho_\infty$ which preserves a totally geodesic copy of $\mathbb
H^k$ in $\mathbb H^m$,  and whose $\mathbb H^k$-component is conjugated to the standard lattice
embedding $i:\Gamma \rightarrow \textup{Isom}(\mathbb{H}^k)<\textup{Isom}(\mathbb H^m)$.
\end{teor}

From which we deduce

\begin{cor}
Suppose $\rho_n:\Gamma  \rightarrow \textup{Isom}(\mathbb{H}^m)$ is a sequence of representations converging to any ideal point of the Morgan--Shalen compactification of $X(\Gamma,\textup{Isom}(\mathbb{H}^m))$. If $k \leq m$ the sequence of volumes $\textup{Vol}(\rho_n)$ must be bounded from above by $\textup{Vol}(M)-\varepsilon$ with $\varepsilon>0$. 
\end{cor}


\vspace{20pt}
Stefano Francaviglia\\
Department of Mathematics,\\
University of Bologna,\\
Piazza di Porta San Donato 5,\\
40126 Bologna,\\
Italy\\
\texttt{stefano.francaviglia@unibo.it}\\
\\
Alessio Savini\\
Department of Mathematics,\\
University of Bologna,\\
Piazza di Porta San Donato 5,\\
40126 Bologna,\\
Italy\\
\texttt{alessio.savini5@unibo.it}\\


\begin{thebibliography}{ABCDE}


\bibitem[Apa90]{Ap} B. N. Apanasov, \emph{Bending and stamping deformations of hyperbolic
    manifolds}, Ann. Glob. Anal. Geom. \textbf{8}(1) (1990), 3--12.
  
\bibitem[BCG95]{besson95:articolo} G. Besson, G. Courtois, S. Gallot,
\emph{Entropies et rigidités des espaces localement symétriques de courbure strictement négative}, Geom. Funct. Anal. \textbf{5} (1995), no. 5, 731--799. 

\bibitem[BCG96]{besson96:articolo} G. Besson, G. Courtois, S. Gallot,
\emph{Minimal entropy and Mostow's rigidity theorems}, Ergodic Theory Dynam. Systems \textbf{16} (1996), no. 4, 623--649.

\bibitem[BCG99]{besson99:articolo}  G. Besson, G. Courtois, S. Gallot,
\emph{A real Schwarz lemma and some applications}, Rend. Mat. Appl. (7) \textbf{18} (1998), no. 2, 381--410.

\bibitem[BCS05]{souto:articolo}  J. Boland, C. Connell, J. Souto,
\emph{Volume rigidity for finite volume manifolds}, Amer. J. Math. \textbf{127} (2005), no. 3, 535--550.

\bibitem[BBI13]{bucher2:articolo} M. Bucher, M. Burger, A. Iozzi,
\emph{A dual interpretation of the Gromov-Thurston proof of Mostow rigidity and volume rigidity for representations of hyperbolic lattices}, Trends in harmonic analysis, 47--76, Springer, Milan, 2013. 

\bibitem[BM96]{burger3:articolo} M. Burger, S. Mozes,
\emph{$CAT(-1)$-spaces, divergence groups and their commensurators}, J. Amer. Math. Soc. \textbf{9} (1996), no. 1, 57--93.

\bibitem[Cal61]{calabi:articolo} E. Calabi,
\emph{On compact Riemannian manifolds with constant curvature}, I Proc. Sympos. Pure Math., Amer. Math. Soc., Providence, R. I., 1961, 155--180.

\bibitem[CS83]{culler:articolo} M. Culler, P. B. Shalen,
\emph{Varieties of group representations and splitting of $3$-manifolds}, Ann. of Math. \textbf{117} (1983), 109--146.

\bibitem[Dun99]{dunfield:articolo}  N. M. Dunfield,
\emph{Cyclic surgery, degrees of maps of character curves, and volume rigidity for hyperbolic manifolds.} Invent. Math. \textbf{136} (1999), no. 3, 623--657.

\bibitem[Fra04]{franc04:articolo} S. Francaviglia,
\emph{Hyperbolic volume of representations of fundamental groups of cusped 3-manifolds}, Int. Math. Res. Not. 2004, no. 9, 425--459.

\bibitem[Fra09]{franc09:articolo} S. Francaviglia,
\emph{Constructing equivariant maps for representations}, Ann. Inst. Fourier (Grenoble) \textbf{59} (2009), no. 1, 393--428. 

\bibitem[FK06]{franc06:articolo} S. Francaviglia, B. Klaff,
\emph{Maximal volume representations are Fuchsian}, Geom. Dedicata \textbf{117} (2006), 111--124.

\bibitem[FP08]{FP} S. Francaviglia, J. Porti,
\emph{Rigidity of representations in $SO(4,1)$
for Dehn fillings on $2$-bridge knots}, Pacific Journal of Mathematics \textbf{238}(2) (2008), 249--274.

\bibitem[GR72]{garland:articolo} R. Garland, M. Raghunathan,
\emph{Fundamental domain for lattices in rank 1 semisimple Lie groups}, Ann. of Math. \textbf{92} (1970), 276--326.

\bibitem[Gui16]{guilloux:articolo} A. Guilloux,
\emph{Volumes of representations and birationality of the peripheral holonomy}, \texttt{https://arxiv.org/abs/1605.05917}.

\bibitem[Kim16]{kim:articolo} S. Kim,
\emph{On the equivalence of the definitions of volume of representations}, Pacific J. Math. \textbf{280} (2016), no. 1, 51--68.

\bibitem[LZ17]{le:articolo} T. T. Q. Lê, X. Zhang,
\emph{Character varieties, A-polynomials and the AJ conjecture}, Algebr. Geom. Topol. \textbf{17} (2017), no. 1, 157--188.

\bibitem[Mor86]{morgan2:articolo} J. W. Morgan,
\emph{Group actions on trees and the compactification of the space of classes of $SO(n,1)$-representations}, Topology \textbf{25} (1986), 1--33.

\bibitem[MS84]{morgan:articolo} J. W. Morgan, P. B. Shalen,
\emph{Valuations, trees and degeneration of hyperbolic structures I}, Ann. of Math., \textbf{120}, 1984, 401--476

\bibitem[Nic89]{Nic89} P.J. Nicholls, \emph{The ergodic theory of discrete groups},
    London Mathematical Society Lecture Note Series, 143, Cambridge University Press (1989).

\bibitem[NZ85]{neumann:articolo} W. D. Neumann, D. Zagier, 
\emph{Volumes of hyperbolic three-manifolds}, Topology \textbf{24} (1985), no. 3, 307--332. 

\bibitem[Rob00]{Rob00} T. Roblin, Thomas, \emph{Sur l'ergodicit\'e rationnelle et les propri\'et\'es
              ergodiques du flot g\'eod\'esique dans les vari\'et\'es
              hyperboliques}, Ergodic Theory Dynam. Systems \textbf{20} (2000), no. 6, 1785--1819. 

\bibitem[Sel60]{selberg:articolo} A. Selberg,
\emph{On discontinuous groups in higher dimension symmetric spaces}, Contributions to Functional Analysis, Tata Institute, Bombay, 1960, 147--164.

\bibitem[Sul79]{Sul79} D. Sullivan, \emph{The density at infinity of a discrete group of hyperbolic
              motions}, Inst. Hautes \'Etudes Sci. Publ. Math., \textbf{50} (1979), 171--202. 
 
\bibitem[Thu81]{thurston:libro} W. Thurston,
\emph{Geometry and topology of $3$-manifolds}, Princeton University Lecture Notes, 1981.
             
\bibitem[Wei62]{weil:articolo} A. Weil,
\emph{On discrete subgroups of Lie groups II}, Ann. of Math. \textbf{75} (1962), 578--602.              
              
\bibitem[Yue96]{Yue96} C. Yue, \emph{The ergodic theory of discrete isometry groups on manifolds of
              variable negative curvature}, Trans. Amer. Math. Soc. (1996), \textbf{12}, 4965--5005.

\end{thebibliography}
\end{document}